\documentclass[a4paper,11pt]{article}

\usepackage{graphicx}
\usepackage{amssymb}
\usepackage{unit}
\usepackage{color}
\usepackage{mathrsfs}
\usepackage{amsmath}
\usepackage{subfigure}
\usepackage{color}
\usepackage{amsthm}
\usepackage{amscd}

\theoremstyle{definition}
\newtheorem{theorem}{Theorem}

\newtheorem{remark}[theorem]{Remark}

\begin{document}
\title{%
On the Spectral Equivalence of Koopman Operators\\
through Delay Embedding
}%
\author{%
Yoshihiko Susuki\footnote{Yoshihiko Susuki is with Department of Electrical and Information Systems, Osaka Prefecture University, Japan, \texttt{susuki@eis.osakafu-u.ac.jp}, \texttt{susuki@ieee.org}}~~~ 
Kyoichi Sako~~~
Takashi Hikihara\footnote{Kyoichi Sako and Takashi Hikihara are with Department of Electrical Engineering, Kyoto University, Japan.}
}%
\date{June 3, 2017}

\maketitle

\begin{abstract}
We provide one theorem of spectral equivalence of Koopman operators of an original dynamical system and its reconstructed one through the delay-embedding technique. 
The theorem is proved for measure-preserving maps (e.g. dynamics on compact attractors) and provides a mathematical foundation of computing spectral properties of the Koopman operators by a combination of extended dynamic mode decomposition and delay-embedding.
\end{abstract}

\section{Introduction}

\par
Applied Koopman operator theory is a developing research area in dynamical systems during past two decades \cite{Marko_CHAOS22}.  
Koopman operator itself, which is a linear, infinite-dimensional operator defined for analysis of nonlinear Hamiltonian systems \cite{Koopman_PNAS17}, is a well-known mathematical object in dynamical systems and ergodic theory: see e.g. \cite{Arnold:1968,Peterson:1983}. 
Recent efforts and progress are made for general (dissipative, non-autonomous, non-stationary, and hybrid)  dynamics with a strong connection to data-driven science and technologies: fluids \cite{Rowley_JFM641,Mezic_ARFM45}, power grids \cite{Susuki_NOLTAIEICE7}, neural activity \cite{Brunton_JNM258}, molecular kinetics \cite{Wu_JCP146}, and so on.  

Extended Dynamic Mode Decomposition (EDMD) is a widely recognized technique for approximating the Koopman operator directly from time-series data \cite{Matt_JNLS25}.  
In this technique, we define a finite set of basis functions on state space of a target dynamical system and apply a technique of DMD \cite{Schmid_JFM656,Tu_JCD1} to the data set encoded via the basis functions from original data on \emph{state dynamics}, that is, a bundle of state trajectories.  
Thereby, it is possible to approximately obtain eigenvalues and eigenfunctions of the Koopman operator as a super-position of the basis functions.  

One question arises when we solve practical problems with EDMD.    
Real-world dynamics are accessible only from a few measurements (observations) and are hard to obtain their mathematical model in state space.  
To overcome this, the so-called delay-embedding \cite{Takens:1981} is a powerful technique to extract information of state dynamics from scalar time-series.     
Indeed, Mezi\'c and Banaszuk \cite{Mezic_PD197} focus on the delay-embedding and propose its extended version---statistical Takens theorem---for constructing ergodic partition of state space directly from scalar time-series of a single observable, which enables us to visualize level sets of eigenfunctions of the Koopman operator.  
Das and Giannakis \cite{Das_SIAMDS17} study the spectra of Koopman operators (semi-groups) with an asymptotic approximation based on integral operators and delay-embedding.  
Rowley and Steyert \cite{Rowley_SIAMDS17} study a fundamental question of how continuous spectra of the Koopman operators for measure-preserving systems are captured with EDMD, where the delay-embedding is addressed.   

As we will indicate below, the Koopman operator naturally appears in the construction of delay-embedding.  
A fundamental question posed here is the spectral equivalence of two Koopman operators for the original and reconstructed (possibly higher-dimensional) dynamical systems. 
In this note, we provide one theorem of the spectral equivalence for measure-preserving maps.  
This provides a mathematical foundation that applying EDMD to reconstructed data does work for estimating the spectrum of the Koopman operator of the original system.  
This result was firstly announced in the International Symposium on Nonlinear Theory and its Applications (NOLTA) \cite{Susuki_NOLTA2016}.

\section{Model System and Definitions}

In this note, we will consider a smooth diffeomorphism (map, dynamical system) $T$ defined on a probability space $(M,\mathfrak{B}_M,\mu)$, where $M$ is a compact metric space, $\mathfrak{B}_M$ is the Borel $\sigma$-algebra of $M$, and $\mu$ a probability measure.  
The diffeomorphism $T: M\to M$ is called \emph{measure-preserving} with $\mu$ if $\mu\circ T^{-1}=\mu$ holds.  
Also, we denote by $\mathcal{F}_M$ a space of observables which correspond to scalar-valued functions defined on $M$, and we denote by $L_2(M)$ the space of functions for which the 2nd power of absolute value is integrable on $M$.  

The \emph{Koopman operator} $\mathbf{U}_T: \mathcal{F}_M\to\mathcal{F}_M$ for the map $T$ is defined as a composition operator with $T$: for $f\in\mathcal{F}_M$, 
\[
\mathbf{U}_Tf = f\circ T.
\] 
The Koopman operator $\mathbf{U}_T$ is linear and infinite-dimensional even if the target dynamical system evolves on a finite-dimensional space.  
For $T$ measure-preserving, $\mathbf{U}_T$ becomes an unitary operator \cite{Arnold:1968,Peterson:1983}.

\section{Delay-Embedding and Spectral Equivalence of Koopman operators}

First of all, we provide the celebrated Takens embedding theorem \cite{Takens:1981} in which we use the notion of Koopman operator $\mathbf{U}_T$.
\begin{theorem}
\label{thm:Takens}
(Takens 1981) Let $M$ be a compact manifold of dimension $m$.  
For pairs $(T,f)$, ${T}: M\to M$ a smooth diffeomorphism and $f: M\to\mathbb{R}$ a 
smooth observable, it is a generic property that the map $\mathit{\Phi}_{(T,f)}: M\to\mathbb{R}^{2m+1}$, 
defined by
\begin{equation}
\mathit{\Phi}_{(T,f)}({x}) := (f({x}), (\mathbf{U}_Tf)(x),\ldots,(\mathbf{U}_T^{2m}f)(x))
\label{eqn:embedding}
\end{equation}
is an embedding; by ``smooth" we mean at least $C^2$.  
\end{theorem}
\begin{remark}
The domain of $\mathit{\Phi}_{(T,f)}$ is $M$, and its image can be a subset of the codomain $\mathbb{R}^{2m+1}$.   
From Theorem~\ref{thm:Takens}, there exists a homeomorphism $\varphi: M\to\mathit{\Phi}_{(T,f)}(M)\subseteq\mathbb{R}^{2m+1}$.  
Thus, the  \emph{reconstructed map} $\tilde{T}: \varphi(M)\to\varphi(M)$ is defined with the following commutative relation:
\begin{equation}
\varphi\circ{T}=\tilde{T}\circ\varphi.
\label{eqn:comm-T}
\end{equation}
The commutative diagram is shown in Figure~\ref{fig:comm-T}.  
The map $\tilde{T}$ is the mathematical object constructed via the delay-embedding (\ref{eqn:embedding}) directly from scalar time-series data, to which we can apply the EDMD.  
\end{remark}

\begin{figure}[t]
\centering
\Large
\[
\begin{CD}
M @> T >> M\\
@V{\varphi}VV @VV{\varphi}V\\
\mathit{\Phi}_{(T,f)}(M) @>{\tilde{T}} >>\mathit{\Phi}_{(T,f)}(M)
\end{CD}
\]
\caption{Commutative diagram for maps}
\label{fig:comm-T}
\end{figure}

\begin{figure}[t]
\centering
\Large
\[
\begin{CD}
\mathcal{F}_M @> \mathbf{U}_{T} >> \mathcal{F}_M\\
@A{\mathbf{C}_\varphi}AA @AA{\mathbf{C}_\varphi}A\\
\mathcal{F}_{\varphi(M)} @>{\mathbf{U}_{\tilde{T}}} >>\mathcal{F}_{\varphi(M)}
\end{CD}
\]
\caption{Commutative diagram for Koopman operators}
\label{fig:comm-U}
\end{figure}

We now state the theorem on spectral equivalence of Koopman operators through the delay-embedding.  
For this, we denote by $\mathbf{U}_{\tilde{T}}$ the Koopman operator defined for the reconstructed map $\tilde{T}$ and by $\mathcal{F}_{\varphi(M)}$ the space of observables defined on $\varphi(M)$ (it becomes compact) on which $\mathbf{U}_{\tilde{T}}$ acts. 
\begin{theorem}
\label{thm:main}
Suppose that ${T}: M\to M$ is measure-preserving. 
For a smooth observable $f\in\mathcal{F}_M=L_2(M)$, consider the reconstructed map $\tilde{T}: \varphi(M)\to\varphi(M)$ through the delay-embedding $\mathit{\Phi}_{(T,f)}$ in (\ref{eqn:embedding}).    
Then, the associated two Koopman operators $\mathbf{U}_T: \mathcal{F}_M\to\mathcal{F}_M$ and $\mathbf{U}_{\tilde{T}}: \mathcal{F}_{\varphi(M)}\to\mathcal{F}_{\varphi(M)}$ are spectrally equivalent.  
\end{theorem}
\begin{proof}
First of all, we define the composition operator with the homeomorphism $\varphi$, denoted by $\mathbf{C}_{\varphi}: \mathcal{F}_{\varphi(M)}\to\mathcal{F}_M$: for $g\in\mathcal{F}_{\varphi(M)}$,
\begin{equation}
\mathbf{C}_{\varphi}g=g\circ\varphi.
\end{equation}
For $g\in\mathcal{F}_{\varphi(M)}$, we have
\begin{eqnarray}
\mathbf{U}_{\tilde{T}}g &=& g\circ\tilde{T} \nonumber\\
&=& g\circ\left(\varphi\circ{T}\circ\varphi^{-1}\right) \nonumber\\
(\mathbf{U}_{\tilde{T}}g)\circ\varphi 
&=& \left(g\circ\varphi\right)\circ{T} \nonumber\\
\mathbf{C}_{\varphi}\left(\mathbf{U}_{\tilde{T}}g\right)
&=&  \left(\mathbf{C}_{\varphi}g\right)\circ{T} \nonumber\\
&=& \mathbf{U}_T\left(\mathbf{C}_{\varphi}g\right).
\end{eqnarray}
Because $g$ is arbitrary, the following commutative relation of the two Koopman operators $\mathbf{U}_T$ and $\mathbf{U}_{\tilde{T}}$ is derived:
\begin{equation}
\mathbf{C}_{\varphi}\circ\mathbf{U}_{\tilde{T}}
= \mathbf{U}_T\circ\mathbf{C}_{\varphi}.
\label{eqn:comm-U}
\end{equation}
The diagram for the Koopman operators is shown in Figure~\ref{fig:comm-U}.  
Note that the direction of vertical arrows used  in Figures~\ref{fig:comm-T} and \ref{fig:comm-U} are opposite for maps (state dynamics) and Koopman operators (observable dynamics).  

Thus, the proof is ended by showing that the composition operator $\mathbf{C}_{\varphi}$ is an unitary operator for the measure-preserving ${T}$.  
Suppose that the measure $\mu$ is preserved under ${T}$, that is, 
\begin{equation}
\mu\circ{T}^{-1}=\mu.
\end{equation}
Then, we see that the new measure $\mu\circ\varphi^{-1}$ is preserved under the reconstructed map $\tilde{T}$ as shown below:  
\begin{eqnarray}
\mu\circ{T}^{-1} &=& \mu \nonumber\\
\mu\circ\left(\varphi^{-1}\circ\tilde{T}^{-1}\circ\varphi\right) &=& \mu \nonumber\\
\left(\mu\circ\varphi^{-1}\right)\circ\tilde{T}^{-1}
&=& \mu\circ\varphi^{-1}.
\end{eqnarray}
With this, we estimate the inner product of observables $g_1, g_2\in\mathcal{F}_{\varphi(M)}=L_2(\varphi(M))$ on the set $\varphi(M)$ as follows:
\begin{eqnarray}
(g_1,g_2)_{L_2(\varphi(M))} &=& \int_{\varphi(M)}
\overline{g_1(y)}g_2(y)d(\mu\circ\varphi^{-1})(y)  \nonumber\\
&=& \int_{M}
\overline{g_1(\varphi(x))}g_2(\varphi(x))d(\mu\circ\varphi^{-1})(\varphi(x)) \nonumber\\
&=& \int_{M}
\overline{(\mathbf{C}_\varphi g_1)(x)}(\mathbf{C}_\varphi g_2)(x)d\mu(x) \nonumber\\
&=& (\mathbf{C}_\varphi g_1,\mathbf{C}_\varphi g_2)_{L_2(M)}
\end{eqnarray} 
That is to say, $\mathbf{C}_\varphi$ preserves the inner-product in $L_2(\varphi(M))$ and $L_2(M)$. 
Also, by definition, $\mathbf{C}_\varphi$ is linear and bounded.  
For $\varphi$ homeomorphism, $\mathbf{C}_\varphi$ is surjective (from Lemma~4.14 in \cite{Eisner:2015}).  
Hence, $\mathbf{C}_{\varphi}$ is an unitary operator.    
From (\ref{eqn:comm-U}) of the Koopman operators, the proof that $\mathbf{U}_T$ and $\mathbf{U}_{\tilde{T}}$ are spectrally equivalent is ended.
\end{proof}

\section{Conclusion}

In this note, we provided the theorem of spectral equivalence of two Koopman operators connected through the delay-embedding.  
The proof itself is simple and might have been reported in literature.  
If the original dynamical system is measure-preserving, then the Takens embedding theorem implies spectral equivalence of the two Koopman operators.  
In this note, we focused on the measure-preserving maps, and needless to say, the result is applicable to dynamics on compact attractors, where it would be useful for nonlinear time-series analysis such as attractor reconstruction.  
This result guarantees that by applying the EDMD to data of reconstructed state dynamics with the delay-embedding, the spectrum of the original Koopman operator can be estimated.  
Associated eigenfunctions and eigendistributions of the Koopman operator are also estimated in the reconstructed state space. 
They are transformed into the original space via the composition operator $\mathbf{C}_\varphi$.  
The $\mathbf{C}_\varphi$ and homeomorphism $\varphi$ are normally unknown for solving real-world dynamics, and thus a new idea is required, e.g. using a prior information on state dynamics.

\section*{Acknowledgements}

The first author (YS) appreciates Professor Igor Mezi\'c (University of California, Santa Barbara) for discussions and comments on this note.  
This work is supported in part by JST CREST \#JPMJCR15K3 and MEXT KAKEN (B) \#15H03964.


\end{document}